\documentclass[psamsfonts, reqno]{amsart}

\usepackage{amscd,amsmath,amssymb,amsfonts,verbatim}
\usepackage{times}
\usepackage[cmtip, all]{xy}
\usepackage{graphicx}
\usepackage[active]{srcltx}
\usepackage{color}
\usepackage{textcomp}

\definecolor{refkey}{gray}{.5}   
\definecolor{labelkey}{gray}{.5} 
\definecolor{Red}{rgb}{1,0,0}

\newcommand{\ra}{\rightarrow}

\newcommand{\surj}{\ra\!\!\!\ra}	
\newcommand{\ol}{\overline}

\newtheorem{theorem}{Theorem}[section]

\newtheorem{lemma}[theorem]{Lemma}

\theoremstyle{definition}

\newtheorem{example}[theorem]{Example}
\newtheorem{definition}[theorem]{Definition}

\newcommand{\CA}{\mbox{$\mathcal A$}}

	\newcommand{\CR}{\mbox{$\mathcal R$}}

\newcommand{\MI}{\mbox{$\mathfrak I$}}

\newcommand{\Spec}{\mbox{\rm Spec\,}}	\newcommand{\hh}{\mbox{\rm ht\,}}
	
\newcommand{\Aut}{\mbox{\rm Aut\,}}

\newcommand{\Um}{\mbox{\rm Um}}		
		
	\newcommand{\E}{\mbox{\rm E}}

\title{Splitting of projective modules over some affine algebras}
\author{Subarna Mandal}
\newcommand{\Addresses}{{
  \bigskip
  \footnotesize

\textsc{Subarna Mandal, Statistics and Mathematics Unit,Indian Statistical Institute,203 BT Road,
            Kolkata,700108 INDIA}\par\nopagebreak
  \textit{E-mail:} Subarna Mandal \texttt{<subarnamandal23@gmail.com>}

\medskip

  }}

\begin{document}
\maketitle

 ~ {Projective modules, affine algebra, unimodular elements.}

 \begin{abstract}
 This article focuses on defining the Euler class group of stably free module and analyzes the condition for existence of unimodular element in corresponding projective  module.This is an improvisation of work by Banerjee and Das( On a question of Nori:Obstructions,improvements and appliaction). If $A$ is a smooth affine algebra over $\overline{F}_p$ and $dim(A)=d\geq 3$ and $n$ be an integer such that $2n\geq d+2$ and P be a projective module of rank $n$ then they defined Euler class group of $P$ and proved that Euler class of $P$ is trivial iff $P$ contains a unimodular element.
We generalize the result to accomodate any stably free module of rank n.

\end{abstract}

\section{Introduction}

The existence of unimodular elements in projective modules is a fundamental problem in the theory of projective modules over commutative rings. Let $R$ be a commutative Noetherian ring and let $P$ be a finitely generated projective $R$ module. An element $p\in P$ is said to be unimodular if there exists an $R$-linear homomorphism $\phi:P\surj R$ such that $\phi(p)=1$. The existence of a unimodular element in $P$ is equivalent to the existence of a free direct summand of rank one; that is,
$P=Q\oplus R$ for some projective $R$-module $Q$. Thus, the problem of determining whether a projective module admits a unimodular element is closely related to the splitting problems for projective modules.

A projective module $P$ of rank $n$ is called stably free if  $P\oplus R^m= R^{n+m}$
for some non-negative integer $m$. The study of unimodular elements in stably free modules has attracted considerable attention.Consequently it is natural to seek algebraic obstructions to the existence of a unimodular element in a stably free module.

Euler class groups provide an important framework for studying this obstruction. For a suitable oriented projective $R$-module $P$ of rank $n$, one can associate an Euler class

$e(P)\in E^n(R)$
The Euler class measures, in an appropriate sense, the obstruction to the existence of a unimodular element of $P$. 
The relationship between the dimension of the underlying ring and the rank of the projective module is particularly important in this article. Various results concerning Euler classes and unimodular elements have been established under stable-range conditions relating $n$ and $d=\dim R$. In particular, inequalities of the form
$2n\geq d+2$ play an important role in the construction and application of the $n$-th Euler class group.

In this direction, Keshari and Tikader studied the existence of unimodular elements in projective modules of rank $n$.Their work gives important results concerning projective modules over polynomial rings and establishes unimodular-element criteria under suitable hypotheses on the rank and dimension. These results provide a significant contribution to the study of the splitting problem for projective modules.

On the Euler-class side, Keshari studied the relationship between Euler classes and unimodular elements of stably free modules. In an appropriate stable range, he showed that the vanishing of the Euler class of a stably free projective module is equivalent to the existence of a unimodular element. More recently, Banerjee and Das developed the $n$th Euler class group in the setting of affine algebras over an algebraically closed field of positive characteristic and applied it to stably free modules arising from unimodular rows. Their results further demonstrate the effectiveness of Euler classes in detecting the existence of unimodular elements.

These developments naturally lead to the study of Euler classes associated with projective modules represented by more general unimodular matrices. In particular, let
$
\sigma\in Um_{r,n+r}(R)
$
be a unimodular matrix, and let $p$ denote the projective module associated with $\sigma$ according to the construction described below. It is natural to ask whether the Euler class associated with $p$ completely determines the existence of a unimodular element of $P$.

The purpose of the present article is to investigate this question. Let $R$ be a ring of dimension $d$ containing a field $K$, and let $n$ be an integer satisfying
$2n\geq d+2$
For the projective module $P$ associated with the unimodular matrix $\sigma$, we study its Euler class
$e(P)\in E^n(R)$
Our main result establishes the following equivalence:
$P$ has unimodular elements means $e(P)=0\in E^n(R)$.

Thus, under the hypotheses of the theorem, the $n$-th Euler class gives a complete obstruction to the existence of a unimodular element in $P$. In particular, the result provides a criterion for the splitting of a free direct summand of $P$ in the admissible rank range $2n\geq d+2$.

The paper is organized as follows. In Section 2, we recall the necessary definitions and preliminary results concerning unimodular matrices, projective modules, orientations, and Euler class groups. In Section 3, we construct the Euler class associated with the projective module determined by $\sigma$ and establish the required properties. In Section 4, we prove the main theorem. Finally, in Section 5, we discuss the consequences and applications of the main result.1. Introduction

The existence of unimodular elements in projective modules is a fundamental problem in the theory of projective modules over commutative rings. Let $R$ be a commutative Noetherian ring and let $P$ be a finitely generated projective $R$-module. An element $p\in P$ is said to be unimodular if there exists an $R$-linear homomorphism $\varphi:P\rightarrow R$ such that $\varphi(p)=1$. The existence of a unimodular element in $P$is equivalent to the existence of a free direct summand of rank one; that is,
$P\simeq Q\oplus R$
for some projective $R$-module $Q$. Thus, the problem of determining whether a projective module admits a unimodular element is closely related to the splitting and cancellation problems for projective modules.

A projective module $P$ of rank $n$ is called stably free if
$P\oplus R^m\simeq R^{n+m}$for some non-negative integer $m$. The study of unimodular elements in stably free modules has attracted considerable attention. Although a stably free module becomes free after adjoining a suitable number of free summands, it need not be free itself. Consequently, it is natural to seek algebraic obstructions to the existence of a unimodular element in a stably free module.

Euler class groups provide an important framework for studying this obstruction. For a suitable oriented projective $R$-module $P$ of rank $n$, one can associate an Euler class
$e(P)\in E^n(R)$
The Euler class measures, in an appropriate sense, the obstruction to the existence of a unimodular element of $P$. The theory of Euler class groups, developed through the work of Bhatwadekar and Sridharan and subsequent contributions by several authors, has provided powerful methods for studying projective and stably free modules.

The relationship between the dimension of the underlying ring and the rank of the projective module is particularly important in this theory. Various results concerning Euler classes and unimodular elements have been established under stable-range conditions relating $n$ and $d=\dim R$. In particular, inequalities of the form
$2n\geq d+2$
play an important role in the construction and application of the $n$-th Euler class group.

In this direction, Keshari and Tikader studied the existence of unimodular elements in projective modules of rank $n$ using monic inversion techniques. Their work gives important results concerning projective modules over polynomial rings and establishes unimodular-element criteria under suitable hypotheses on the rank and dimension. These results provide a significant contribution to the study of the splitting problem for projective modules.

On the Euler-class side, Keshari studied the relationship between Euler classes and unimodular elements of stably free modules. In an appropriate stable range, he showed that the vanishing of the Euler class of a stably free projective module is equivalent to the existence of a unimodular element. More recently, Banerjee and Das developed the $n$-th Euler class group in the setting of affine algebras over an algebraically closed field of positive characteristic and applied it to stably free modules arising from unimodular rows. Their results further demonstrate the effectiveness of Euler classes in detecting the existence of unimodular elements.

These developments naturally lead to the study of Euler classes associated with projective modules represented by more general unimodular matrices. In particular, let
$\sigma\in Um_{r,n+r}(R)$
be a unimodular matrix, and let $P$ denote the projective module associated with $\sigma$ according to the construction described below. It is natural to ask whether the Euler class associated with $P$ completely determines the existence of a unimodular element of $P$ in the stable range.

The purpose of the present paper is to investigate this question. Let $R$ be a ring of dimension $d$ containing a field $K$, and let $n$ be an integer satisfying
$2n\geq d+2$.
For the projective module $p$ associated with the unimodular matrix $\sigma$, we study its Euler class
$e(P)\in E^n(R)$
Our main result establishes the following equivalence:
$P\text{ has a unimodular element}
\quad\Longleftrightarrow\quad
e(P)=0\in E^n(R)$

Thus, under the hypotheses of the theorem, the $n$-th Euler class gives a complete obstruction to the existence of a unimodular element in $P$. In particular, the result provides a criterion for the splitting of a free direct summand of $P$ in the admissible rank range $2n\geq d+2$.

The paper is organized as follows. In Section 2, we recall the necessary definitions and preliminary results concerning unimodular matrices, projective modules, orientations, and Euler class groups. In Section 3, we construct the Euler class associated with the projective module determined by $\sigma$ and establish the required properties. In Section 4, we prove the main theorem. Finally, in Section 5, we discuss the consequences and applications of the main result.

\section{Preliminaries}
Before going to the next result we recall the following definitions.

Let $A$ be a ring with $1$ and $P$ be a projective $A$-module.
\begin{enumerate}

        \item A row vector $(a_0,...,a_n)\in A^{n+1}$ of length $n+1$ is said to be a unimodular row of length $n+1$, if there exists $(b_0,...,b_n)\in A$ such that $\sum_{i=0}^{n}a_ib_i=1$. We will denote $\Um_{n+1}(A)$ as the set of all unimodular row vectors of length $n+1$ over the ring $A$.
		\item The stable range (denoted by $sr(A)$) of $A$ is the smallest natural number $r$ with the property that for any $(u_1,...,u_r)\in \Um_r(A)$, there exist $\lambda_i\in A$ such that $(u_1+\lambda_1u_{r+1},...,u_r+\lambda_ru_{r+1})\in Um_r(A)$.
		\item $P$ is said to have a unimodular element if there exists an $A-$linear surjection $\phi:P\surj A$. In this case an element $v\in P$ is said to be a unimodular element of $P$ if $v\in \phi^{-1}(1)$. The set of all unimodular elements of $P$ is denoted by $\Um(P)$.
	\item We say that the projective stable range of $A$ (notation: $psr(A)$) is $n$ if $n$ is the least positive integer such that for any projective $A-$module $P$ of rank $n$ and $(p,a)\in \Um(P\oplus A)$, there exists $q\in P$ such that $p+aq\in \Um(P)$.
	\item Let $P$ be such that $\Um(P)\not=\phi$. We choose $\phi\in P^*$ and $p \in P$ such that $\phi(p)=0$. We define an endomorphism $\phi_p$ as the composite $\phi_p:P\to A\to P$, where $A\to P$ is the map sending $1\to p.$ Then by a transvection we mean an automorphism of $P$, of the form $1+\phi_p$, where either $\phi\in \Um(P^*)$ or $p\in \Um(P)$. By $\E(P)$  we denote  the subgroup of $\Aut(P)$ generated by all transvections.

\end{enumerate}

\section{The Euler class group and the Euler class}

Let $A$ be an affine algebra over $\overline{\mathbb{F}}_p$ such that $\dim(A)=d.$ Henceforth we shall assume that $d\ge 3$ and $n$ is an integer satisfying $2n\ge d+2$.
\par Let $J$ be an ideal of height $n$ such that $J/J^2$ is generated by $n$ elements. Let $\alpha$ and $\beta$ be two surjections from $(A/J)^n$ to $J/J^2.$ We say that $\alpha$ and $\beta$ are related if there exists an elementary automorphism $\sigma $ of $(A/J)^n$ such that $\alpha\sigma=\beta$. This is an equivalence relation on the set of surjections from $(A/J)^n$ to $J/J^2$. Let $[\alpha]$ denote the equivalence class of $\alpha$. If $(\ol{a}_1,...,\ol{a}_n)$ generate $J/J^2$, we obtain a surjection $\alpha:(A/J)^n\surj J/J^2$, sending $\ol{e}_i $ to $\ol{a}_i$. We say $[\alpha]$ is given by the set of generators $\ol{a}_1,...,\ol{a}_n$ of $J/J^2$.

\par 
\begin{definition}
Let $G$ be the free abelian group on the set $B$ of pairs $(J,\omega_J)$, where:\\
$(i)$ $J\subset A$ is an ideal of height $n$,\\
$(ii)$ $\Spec(A/J)$ is connected,\\ 
$(iii)$ $J/J^2$ is generated by $n$ elements and \\
$(iv)$ $\omega_J:(A/J)^n\surj J/J^2$ is an equivalence class of surjections $\alpha:(A/J)^n\surj J/J^2$.

\par Let $J\subset A$ be a proper ideal. Then $J=J_1\cap J_2\cap ...\cap J_r$, where $J_i$'s are proper, pairwise co-maximal and $\Spec(A/J_i)$ is connected for all $i=1,...,r$. It was proved in \cite[Lemma 4.1]{BR3} that such a decomposition is unique. We shall say that $J_i$ are the connected components of $J$.
\par Let $J\subset A$ be an ideal of height $n$ such that $J/J^2$ is generated by $n$ elements. Let $J=\cap J_i$ be the decomposition of $J$ into its connected components. Then note that  for every $i$, $\hh(J_i)=n$ and by the Chinese Remainder Theorem  $J_i/J_i^2$ is generated by $n$ elements. Let $\omega_J:(A/J)^n\surj J/J^2$ be a surjection. Then in a natural way $\omega_J$ gives rise to surjections $\omega_{J_{i}}:(A/J_i)^n\surj J/{J_i}^2$. We associate to the pair $(J,\omega_J)$, the element $\sum (J_i,\omega_{J_i})$ of $G$.
\par Let $H$ be the subgroup of $G$ generated by the set $S$ of pairs $(J,\omega_J)$, where  $\omega_J:(A/J)^n\surj J/J^2$ has a surjective lift and denote it by $\theta:A^n\surj J $. Then we define the quotient group $G/H$ as the $n$-th Euler class group of $A$ and denote it by ${E^n(A)}$.
\end{definition}
\begin{theorem}
Let $A$ be an affine algebra over  $\overline{\mathbb{F}}_p$ of dimension $d$ and $n$ be an integer satisfying $2n\ge d+2$. Let $I\subset A$ be an ideal of height $n$ be such that $I/I^2$ is generated by $n$ elements and $\omega_I:(A/I)^n\surj I/I^2 $ be an equivalence class of surjections. Suppose that the image of $(I,\omega_I)$ is zero in the Euler class group $E^n(A)$ of $A$. Then $I$ is generated by $n$ elements and $\omega_I$ can be lifted to a surjection $\theta:A^n\surj I$
\end{theorem}
To prove Theorem , we shall need the following lemma whose proof can be found in \cite[lemma4.1]{MPHIL}

\begin{lemma}
Let $F$ be a
free abelian group with basis $\{e_i\}_{i\in\MI} $ and $\sim$ be an equivalence relation on $\{e_i\}_{i\in \MI}$. Define $x \in F$ to
be ``reduced'' if $x = e_1 + . . . + e_r$ and $e_i\not= e_j$ ($i \not= j$). For $x\in F$ with $x = e_1 + . . . + e_r$ we define ``support" of $x$ to be the set $supp(x) = \{e_1,...,e_r\}$. Define, $x \in F$ to be ``nicely
reduced' if $x = e_1 + . . . + e_r$ and $e_i\not= e_j$ ($i \not= j$) and such that no $e_i$ belongs to the equivalence class of other $e_j$ for $i,j=1,...,r$ and $i \not= j$. Let $S \subset F$
be such that:
\begin{enumerate}

\item Every element of $S$ is nicely reduced;
\item  Let $x,y\in F$ such that $x, y$ and $ x + y$ are nicely reduced. Then if any two of $x,y,x+y$ are in $S$, then so is the third one;
\item  $x \in F$ and $x \not\in  S$ be nicely reduced. Let $K \subset \MI$ be a finite set. Then there exists $y \in  F$ such that:
\begin{enumerate}

\item $y$ is nicely reduced;
\item $x + y \in S$;
 \item  $y + e_k$ is nicely reduced for all $k \in K$.
\end{enumerate}

\end{enumerate}
Let $ H$ be a subgroup of $F$ generated by the set $S$ . Then if $x \in H$ is nicely reduced then $x \in S$.

\end{lemma}

Proof of  Theorem $3.2$,We take $F$ to be the free abelian group generated by the set $B$, as defined in $3.1$ . Define
a relation `$\sim$' on B as $(J,\omega_J)\sim (I,\omega_I)$ if $I=J$. Then it is an equivalence relation.
\\
Let $S\subset G$ be as in $3.1$. In view of the above lemma, it is enough to show that the
three conditions in Lemma $3.3$ are satisfied. Condition $(i)$ is clear, almost from the definition.
The addition and subtraction principles will yield condition $(ii)$.  Finally,
applying the moving lemma \cite[Corollary 2.4]{BR3} , it is clear that $(iii)$ is also satisfied.

\section{Euler class of stably free module $P$}
Let $A$ be a smooth affine algebra over $\overline{F_p}$, of dimension $d\geq 3$.Take n be an integer such that 2n$\geq d+2$. We shall assign the stably free module P of rank n to its 'n-th Euler class' group and show that P has a unimodular element if and only if it's Euler class is trivial.

Let $ r\geq 1$ and $Um_{r,n+r} (A)$ be the set of all $r \times (n+r)$ matrices $\sigma$ in $M_{r,n+r}(A)$ which  has a right inverse, i.e.  there exists $\tau \in M_{n+r,r}(A)$ such that $\sigma\circ \tau=Id_r.$
For any $\sigma \in Um_{r,n+r}(A),$ we have an exact sequence 
$$ 0 \rightarrow A^r \rightarrow^\sigma\, A^{n+r} \rightarrow P \rightarrow 0 $$ 
where $\sigma(v) = v\sigma$ for $v \in A^{r}$ and $P$ is a stably free $A$-module of rank $n.$ Hence every element of $Um_{r,n+r} (A)$ corresponds to a stably free $A$-module of rank $n$ and conversely, any stably free $A$-module $P$ of rank $n$ will give rise to an element of $Um_{r,n+r}(A)$ for some $r.$ We will define a map 
$$e:Um_{r,n+r}(A) \rightarrow E^n(A)$$ which is a natural generalization of the map $Um_{n+1}(A) \rightarrow E^n (A)$ defined in \cite{BRS}.
Let $\sigma$  be an element of $Um_{r,n+r}(\CA),$ then 
\[
\sigma=
  \begin{bmatrix}
    a_{1,1} & \cdots & a_{1,n+r} \\
    \vdots  &  \cdots & \vdots \\
    a_{r,1} & \cdots & a_{r,n+r} \\
  \end{bmatrix}
\]
Let $e_1,\ldots,e_{n+r}$ be the standard basis of $A^{n+r}$ and let 
$$ P = A^{n+r}/\left(\sum^{n+r} _{i=1} a_{1,i}e_i,\ldots, \sum _{i=1} ^{n+r} a_{r,i}e_i \right)A. $$
Let  $p_1,\ldots,p_{n+r}$  be the images of  $e_1,\ldots,e_{n+r}$ respectively in $P.$ Then 
$$ P = \sum _{i=1} ^{n+r} A p_i  ~\mbox{ with relations } ~\sum_{i=1}^{n+r} a_{1,i}p_i = 0 , \ldots ,\sum_{i=1}^{n+r} a_{r,i}p_i = 0 .$$
To the triple $(P,(p_1,\cdots,p_{n+r}),\sigma),$ we associate an element $e(P,(p_1,\cdots,p_{n+r}),\sigma)$ of $E^n(A)$ as follows:
Let $\lambda :P \surj J$ be a generic surjection, i.e. $J \subset A$ is an ideal of height $n.$ Since $P \oplus A^r = A^{n+r}$ and dim $A/J = d-n \leq n-2$, we get $P/JP$ is a free $A/J $ module of rank $n$, by Bass \cite{bass}. Since $J/J^2$ is surjective image of $P/JP,$ $J/J^2$ is generated by $n$ elements.
Let ``bar" denote reduction modulo $J$. By Bass \cite{bass} , there exist $\Theta \in E_{n+r}(\ol {A})$ such that $$[\ol{a_{1,1}} ,\ldots ,\ol{a_{1,n+r}}]\Theta = [1,0,\ldots,0]$$ i.e.
the first row of $\Theta ^{-1}$ is $[\ol{a_{1,1}} ,\ldots ,\ol{a_{1,n+r}}]$. Let $\ol{\sigma}\circ\Theta$ be given by 
\[
\ol \sigma\circ \Theta=
  \begin{bmatrix}
    1       & 0 &0 &0   \\
   \ol{b_{2,1}} & \ol{b_{2,2}} &\cdots & \ol{b_{2,n+r}} \\
   \vdots  &  \vdots &  \vdots & \vdots \\
   \ol{b_{r,1}} & \ol{b_{r,2}} &\cdots & \ol{b_{r,n+r}} 
  \end{bmatrix}.
\]
Note that $[\ol{b_{2,2}} ,\ldots , \ol{b_{2,n+r}}] \in Um_{n+(r-1)} (\ol {A}).$ By Bass \cite{bass},  there exist $\Theta_1 \in E_{n+(r-1)}(\ol {A})$ such that $$[\ol{b_{2,2}} ,\ldots ,\ol{b_{2,n+r}}]\Theta_1 = [1,0,\ldots,0].$$ If $\Phi \in E_m(A)$, then  $
\begin{bmatrix}
Id_t &0 \\
0 &\Phi \\
\end{bmatrix}
\in E_{m+t}(A)$.
Let
\[
\ol{\sigma}\circ \Theta\circ \Theta_1 =
 \begin{bmatrix}
    1       & 0 &0 &0   \\
   \ol{b_{2,1}} & 1&\cdots & 0\\
    \ol{b_{3,1}}& \ol{c_{3,2}} &\cdots & \ol{c_{3,n+r}} \\
   \vdots  &  \vdots &  \vdots & \vdots \\
   \ol{b_{r,1}} & \ol{c_{r,2}} &\cdots & \ol{c_{r,n+r}} \\
  \end{bmatrix}
\]
Counting in this way, we get $\tilde{\Theta} \in E_{n+r} (\ol {A})$ such that 
\[
\ol{\sigma}\circ \tilde{\Theta} =
 \begin{bmatrix}
    1       & 0 &0 &0 &\cdots &0   \\
   \ol{b_{2,1}} & 1 &0 &0  &\cdots & 0\\
    \ol{b_{3,1}}& \ol{c_{3,2}} &1 &0 &\cdots & 0 \\
   \vdots  &  \vdots &  \vdots & \vdots \\
   \ol{b_{r,1}} & \ol{c_{r,2}}&\cdots & \ol{d_{r,r-1}} &1 &\cdots & 0 \\
  \end{bmatrix}.
\]
We can find elementary matrix $\Psi \in E_{n+r} (\ol { A})$ such that 
$$\ol{\sigma}\circ \tilde{\Theta} \circ \Psi = [Id_r,\underline{0}]$$
 where $\underline{0}$ is $r \times n$ zero matrix. 
Let $\Delta = (\tilde{\Theta}\circ \Psi)^{-1} \in E_{n+r}(\ol {A}),$ then $\ol{\sigma}$ is the first $r$ rows of $\Delta$, i.e, $\ol{\sigma}$ can be completed to an elementary matrix $\Delta.$ Since
$$  \sum_{i=1}^{n+r} a_{1,i}p_i = 0 , \ldots ,\sum_{i=1}^{n+r} a_{r,i}p_i = 0 .$$ we get $$ \Delta [\ol{p_1},\ldots,\ol{p_{n+r}}]^{t} = [0,\ldots,0,\ol{q_1},\ldots,\ol{q_n}]^t, $$ where $t$ stands for transpose.
Thus $(\ol{q_1},\ldots,\ol{q_n})$ is a basis of the free module $P/JP.$ 

Let $\omega_J : (\CA/J)^n \surj J/J^2$
 be the surjection given by the set of generators of $\ol{\lambda(q_1)},\cdots,\ol{\lambda(q_n)}$ of $J/J^2.$
We define $$e(P,(p_1,\cdots,p_{n+r}),\sigma) = (J,\omega_J) \in E^n(\CR).$$ 
We need to show that $e(P,(p_1,\cdots,p_{n+r}),\sigma)$ is independent of the choice of the elementary completion of $\ol{\sigma}$ and the choice of the generic surjection $\lambda.$

We begin with the following result which shows that  $e(P,(p_1,\cdots,p_{n+r}),\sigma)$ is independent of the choice of the elementary completion $\ol{\sigma}$.

\begin{lemma} \label{2.1}
Suppose $\Gamma \in E_{n+r}(\ol {\CA})$ is chosen so that  its first $r$ rows are $\ol{\sigma}$. Let  $ \Gamma [\ol{p_1},\ldots,\ol{p_{n+r}}]^{t} = [0,\ldots,0,\ol{q_1'},\ldots,\ol{q_n'}]^t.$ Then  there exist $\Psi \in E_n(\ol {\CA})$ such that $\Psi[\ol{q_1},\ldots,\ol{q_n}]^t = [\ol{q_1'},\ldots,\ol{q_n'}].$
\end{lemma}

\begin{proof}
The Matrix $\Gamma \circ \Delta ^{-1} \in E_{n+r}(\ol {\CA})$ is such that its first $r$ rows are $[Id_r,\underline{0}].$ Therefore, there exists $\Psi \in SL_n(\ol {\CA}) \cap E_{n+r}(\ol {\CA})$ such that $\Psi[\ol{q_1},\ldots,\ol{q_n}]^t = [\ol{q_1'},\ldots,\ol{q_n'}].$ 
Since $n > dim \ol {\CA} +1$, by Suslin-Vaserstein \cite{SV}, $\Psi \in E_n(\ol {\CA}).$
$\hfill\square$
\end{proof}
\medskip

Let $\tilde{\omega_J} : (\CA/J)^n \surj J/J^2$ be the surjection given by the set of generators $\ol{\lambda(q_1')},\cdots,\ol{\lambda(q_n')}$ of $J/J^2.$
Then, by (\ref{2.1}), $(J,\omega_J) = (J,\tilde{\omega_J})\in E^n (\CR).$ Thus for a given surjection $\lambda :P \surj J,$ the element $e(P,(p_1,\cdots,p_{n+r}),\sigma)$ is independent of the choice of the elementary completion of $\ol{\sigma}.$

Now we have to show that $e(P,(p_1,\cdots,p_{n+r}),\sigma)$ is independent of the choice of the generic surjection $\lambda.$ In other words, we can state as following:
\begin{theorem}
 
$\lambda' :P \surj J'$ is another generic surjection, where $J'$ is an ideal of $\CA$ of height $n$ and  $\omega_J' : (\CA/J')^n \surj J'/J'^2$ is a surjection obtained as above by completion of $\sigma$ modulo $J'$ to an element of $E_{n+r}(\CA/J'),$ then  $(J,\omega_J) = (J',\omega_J')\in E^n (\CA).$ 

\end{theorem}
\begin{proof}
[By lemma 5.1,BRS]There exists an ideal $I$ of $A[T]$ of height n and a surjection $\lambda[T]:P[T]\surj I$ such that $\lambda(0)= 0$ and $\lambda(1)=\lambda'$
Take $N= (I\cap A)^2$
Then $ n-1 \leq ht(N) \leq n$ and by our assumtion $dim(A/N)\geq 2$ . 
Then by Suslin $Sr (A/N)\leq dim (A/N)\leq d-n+1 \leq n+1$
Hence,there exists $\Delta \in E_{n+r}(\ol {\CA})$ such that $\Delta [\ol(p_1),\cdots,\ol(p_{n+r})]^t = [0,0,\cdots,\ol(q_1),\ol(q_2),\cdots,\ol(q_n)]^t$
Where $(\ol(q_1),\ol(q_2),\cdots,\ol(q_n))$ forms a basis of the free $(A/N)$ module $(P/NP)$ and also a basis of free $A[T]/I$ module $P[T]/IP[T]$ .
Therefore as earlier,we obtain [${\ol\lambda(T)(\ol{q_1},\cdots,\ol\lambda(T)(\ol{q_n})}$] as a set of generators of $I/I^2$ and setting $T=0$ and $T=1$ we ontain generators of $J/J^2$ and $J'/J'^2$ respectively and hence the corresponding surjections $\omega_J' : (\CA/J')^n \surj J'/J'^2$ and $\omega_J' : (\CA/J')^n \surj J'/J'^2$.
There exists $K\subset R$ of $htn$ and comaximal with $J$ and $J'$ and surjection $\omega_k:(CA/K)^n\surj K/K^2$ such that $(J,\omega_j) + (K,\omega_k) = (J',\omega_J') + (K,\omega_K)$ Hence,the assignement is well defined.

\end{proof}

\begin{theorem}
 Suppose that there exists $\epsilon' \in E_{n+1}(A)$ with  $\ol \epsilon'\ol e_1^T=\ol v^T$ such that $\ol p \ol \epsilon' =(\ol 0, \ol u'_1,...,\ol u'_n)$. Let $\omega'_J$ be a surjection induced by $\lambda$, sending each $\ol u'_i \to \ol {\lambda (u'_i)}$, for $i=1,...,n$. Then  $(J,\omega_J)=(J,\omega'_J)$. 
\end{theorem}

Theorem follows from the following lemma.
\begin{lemma}
Suppose that there exists $ \epsilon'\in E_{n+1}(A)$ with $\ol \epsilon'\ol e_1^T=\ol v^T$ such that  $\ol p \ol \epsilon' =(\ol 0, \ol u'_1,...,\ol u'_n)$. Then there exists $\theta \in E_n(A/J)$ such that $(\ol u_1,...,\ol u_n)\theta =(\ol u'_1,...,\ol u'_n)$.
\end{lemma}

\par Hence the assignment of $(J,\omega_J)$ associated to $(P,v,p)$ is well defined and we shall denote it by $e(v)\in E^n(A)$. From the definition of $E^n(A)$ it follows that, for any two unimodular rows $u,v\in \Um_{n+1}(A)$, if $u\sim_{E_{n+1}(A)}  v$, then $e(v)=e(u)$. Therefore, we obtain a map $$e:\Um_{n+1}(A)/E_{n+1}(A)\to E^n(A).$$


\begin{theorem} \label{main61}
Let $A$ be a ring of dimension $d$ containing a field $k$ and $n$ be an integer with
 $2n \geq d+2$. Suppose $\sigma$ and P be defined as before. Then P has a unimodular element if and only if $e(P)=0$ in $E^n(A)$.
\end{theorem}

\begin{proof}
First we assume that $P$ has a unimodular element,say $ u\in Um(P)$. Therefore,we have $P= Q\oplus Au$. Let $\lambda:P\surj J$ be a surjection,where $J$ be an ideal if height n. By \cite[Corollary 2.13 ]{SMBB3}without loss of generality we may assume that, ht$(\lambda(Q))=n-1$. Let $N=\lambda(Q)$ and bar denote going modulo $N$. Since, $n+1\geq d-n+3=dim(A/N)+2$ ,then by elementary transposition described in previous theorem 
We can find elementary matrix $\Delta \in E_{n+r} (\ol { A})$ such that $$ \Delta [\ol{p_1},\ldots,\ol{p_{n+r}}]^{t} = [0,\ldots,0,\ol{q_1},\ldots,\ol{q_n}]^t, $$ where $t$ stands for transpose.
Thus $(\ol{q_1},\ldots,\ol{q_n})$ is a basis of the free $A/N$ module $P/NP.$
Let $\ol{u}=\sum_{i=1}^{n}\ol{ a_{i}}\ol{q_i}$
Since $\ol{u}$ is a unimodular element of the free $A/N$ module $P/NP$,we have $(\ol{a_1},\ol{a_2},..\ol{a_n})\in Um_n(A/N)$. Since $dim(A/N)=d-n+1\geq 2$ We have, $n\geq dim(A/N)+1\geq sr(A/N)+1.$
This implies that $$(\ol{a_1},\ol{a_2},....,\ol{a_n}) \sim E_{n+r}(A/N)(1,0,0,....0)$$
Hence,$(\ol{q_1},\ol{q_2},...\ol{q_n})$ can be taken by elementary automorphism to a basis $(\ol{u},\ol{t_2},...,\ol{tn})$ where $\ol{t_i} \in Q/NQ$, for $i=2,...,n$. Let Let $\ol \lambda(t_i)=\ol b_i$, for $i=2,...,n$. Therefore, we get $J=(b_2,...,b_n,c)+J^2$ and $N=(b_2,...,b_n)+N^2$, where $c=\lambda(u)$. Then by \cite{NMKL} there exists $e\in N^2$ such that $e(1-e)\in (b_2,...,b_n)$. Let $N=(b_2,...,b_n,e)$ and $J=(N,c)=(b_2,...,b_n,e+(1-e)c)$. Therefore, we obtain a lift of $J=(b_2,...,b_n,c)+J^2$, which implies that $e(v)=0$ in $E^n(A)$.

Conversely, we assume that $e(\sigma)=0$ in $E^n(A)$.Let $\lambda:P\surj J$,be a surjection,where $J\subset A$ be an ideal of height $n$. let bar denote going modulo J. Since$n+1 \geq dim(A/J)+2$ there exist $\alpha \in E^n(A)$ such that $$ \Delta [\ol{p_1},\ldots,\ol{p_{n+r}}]^{t} = [0,\ldots,0,\ol{q_1},\ldots,\ol{q_n}]^t, $$ where $t$ stands for transpose.Thus $(\ol{q_1},\ldots,\ol{q_n})$ is a basis of the free module $P/JP$. Hence we get $ J=(\lambda(q_1),\lambda(q_2),...\lambda(q_n))+J^2 $
 Since $e(\sigma)=0$ we can have a lift for $J$.
 we can obtain $J=(b_1,b_2,...,b_n)$ with $\lambda(q_i)-b_i \in J^2$ for $i=1,2,..n$.
Using E-E,we can find $c_1,c_2,..,c_{n-1} \in A $ such that $ht(b_1+c_1b_n,b_2+c_2b_n,...,b_{n-1}+c{n_1}b_n)=n-1$  
\par Let $d_i=b_i+c_ib_n$ for $i=1,2,..,n-1$ and $d_n=b_n$
\par Let $I=(d_1,d_2,...d_n+T)$ and 'tilde' denote going modulo $I$.Then we have $P[T]=(A[T])/(p_1,p_2,...p_(n+r))$.(check the definition of P[T]) We see that 
$\dim (A[T]/I)=\dim(A/(d_1,d_2,...d_(n-1))=d-n+1\leq(n-1)\leq(n+1)$ and hence we can find $\tau(T)\in E_(n+r)(A[T])$ such that .Therefore,$[\widetilde(p_1),..\widetilde{P_(n+r)}]\widetilde{\tau(T)}=[0,0,..\widetilde{q_1(T)},...\widetilde{q_n(T)}]$ where $\widetilde{q_1(T)},...\widetilde{q_n(T)}]$ forms a basis of the free $A[T]/I$ module $P[T]/IP[T]$. Evaluating at $T=0$ we get$[\widetilde(p_1),..\widetilde{P_(n+r)}]\widetilde{\tau(0)}=[0,0,..\widetilde{q_1(0)},...\widetilde{q_n(0)}]$ with $\tau(0) \in E_(n+r)(A)$.
\par Therefore by theorem $8.6$ We can find $\theta_1 \in E_n(A/J)$ such that,$(\ol{q_1},..\ol{q_n})=(\ol{q_1(0)},...\ol{q_n(0)})\theta_2$
\par Note that $\ol{\lambda(q_i+c_iq_n)}=\ol{d_i}$ for all $i=1,...,n-1$ and $\ol{\lambda(q_n)}=\ol{d_n}$. Also note that $A/J\cong (A[T]/I)/(t)$, where $t=\widetilde{T}$. Therefore, the map $A[T]/I\surj A/J$ is surjective and so and so is $E_{n}(A[T]/I)\surj E_{n}(A/J)$. Hence we can find $\tau(T) \in E_n(A[T])$, which is a lift of $\theta_2$. Let 
$$(\widetilde{q_1(T)},...,\widetilde{q_n(T)})\widetilde{\tau(T)}=(\widetilde{w_1(T)},...,\widetilde{w_n(T)}).$$
Since $(\widetilde{q_1(T)},...,\widetilde{q_n(T)})$ is a basis of the free $A[T]/I$-module $P[T]/IP[T]$, $\{\widetilde{w_1(T)},...,\widetilde{w_n(T)}\}$ is also a basis. Define a surjection $\theta:P[T]/IP[T]\surj I/I^2 $, sending $\widetilde{w_i(T)}\to {d_i}$ for $i=1,...,n-1$ and $\widetilde{w_n(T)}\to {d_n+T}$. Then $\widetilde{\tau(0)}=\theta_2$, implies that $\theta(0)=\ol \lambda$. Since we have $d\ge n+2$ this gives us  $\dim(A[T]/I)+1\le d-n+2\le n$, by Corollary \cite[5.3]{sbmkd} there exists a surjective lift $\Theta$ of $\theta :P[T]\surj  I$ which matches at $T=0$. Since $T+d_n\in I$, setting $T=1-d_n$, we obtain a surjection from $\gamma=\Theta(1-d_n): P\surj A$, which completes the proof.\qed

\end{proof}

\Addresses

\bibliographystyle{abbrv}


\begin{thebibliography}{10}
	
	
	
	\bibitem{bass}
	H.~{B}ass.
	\newblock Algebraic $K$-theory.
	\newblock W. A. Benjamin Inc. 1968.
	
	\bibitem{BR}
	S.~M. Bhatwadekar and R.~Sridharan.
	\newblock Projective generation of curves in polynomial extensions of an affine
	domain and a question of {N}ori.
	\newblock {\em Inventiones Mathematicae}, 133(1):161--192, jun 1998.
	
	\bibitem{BRS}
	S.~M. Bhatwadekar and R.~Sridharan.
	\newblock Zero cycles and the {E}uler class groups of smooth real affine
	varieties.
	\newblock {\em Inventiones Mathematicae}, 136(2):287--322, apr 1999
	\bibitem{BR3}
	S.~M. Bhatwadekar and R.~Sridharan.
	\newblock On {E}uler classes and stably free projective modules.
	\newblock In {\em Algebra, arithmetic and geometry, {P}art {I}, {II} ({M}umbai,
		2000)}, volume~16 of {\em Tata Inst. Fund. Res. Stud. Math.}, pages 139--158.
	Tata Inst. Fund. Res., Bombay, 2002.
	\bibitem{soumimanoj}
	M.~K. Keshari and Soumi.~Tikader
	\newblock On a question of Moshe Roitman and Euler class of stably free module.
	\newblock In{\em Journal of Algebra 614 (2023) 848–866}
	
	\bibitem{MPHIL}
	M.~K. Keshari.
	\newblock Euler class group of a {N}oetherian ring.
	\newblock {\em M-Phil Thesis-2001 (52 page).}, Aug. 2001.
	
	\bibitem{SMBB3}
	S.~M. Bhatwadekar and R.~Sridharan.
	\newblock The {E}uler class group of a {N}oetherian ring.
	\newblock {\em Compositio Mathematica}, 122(2):183--222, 2000.
	
	\bibitem{NMKL}
	N.~Mohan~Kumar.
	\newblock Complete intersections.
	\newblock {\em Journal of Mathematics of Kyoto University}, 17(3):533--538,
	1977.
	
	\bibitem{sbmkd}
	S.~Banerjee. and Mrinal.~Kanti.~Das.
	\newblock On the question of nori;Obstructions,improvement and applications.
	\newblock {\em Journal of Algebra}, 635:271--299, 2023.
	
	\bibitem{SV}
	L.~N. Vaser{\v{s}}te{\u{\i}}n and A.~A. Suslin.
	\newblock {S}erre's problem on projective modules over polynomial rings, and
	algebraic {K}-theory.
	\newblock {\em Mathematics of the {USSR}-Izvestiya}, 10(5):937--1001, oct 1976.
	\end{thebibliography}

\end{document}